\documentclass[12pt]{amsart}
\usepackage{tikz-cd}
\usepackage{fancyhdr} 
   
\usepackage{latexsym, amssymb, amscd, amsfonts,pstricks,enumitem,amsmath, young}
\usepackage{latexsym, amssymb, amscd, amsfonts,pstricks,enumitem,amsmath}
\usepackage{amsmath,amsthm,amssymb,mathrsfs}
 
\usepackage[all]{xypic}
\usepackage{ifthen}
\usepackage{longtable}
 
\renewcommand{\subsection}[1]{\vspace{3mm}\refstepcounter{subsection}\noindent{\bf \thesubsection. #1.} }

\newcommand{\np}{\vspace{3mm}\refstepcounter{subsection}\noindent{\bf \thesubsection. }}

\renewcommand{\subsubsection}[1]{\vspace{3mm}\refstepcounter{subsubsection}\noindent{\bf \thesubsubsection. #1.} }

\numberwithin{equation}{section}

\renewcommand{\geq}{\geqslant}
\renewcommand{\leq}{\leqslant}
\newcommand{\Osh}{{\mathcal O}}                        

\renewcommand{\H}{\mathrm{H}}                          
                            
\newcommand{\K}{\mathrm{K}}                            
\newcommand{\cchar}{\operatorname{char}}

\newcommand{\Ish}{\mathcal{I}}

\renewcommand{\div}{\operatorname{div}}

\newcommand{\kk}{\mathbf{k}}

\newcommand{\N}{\operatorname{N}}
\newcommand{\Vol}{\operatorname{Vol}}
\newcommand{\spec}{\operatorname{Spec}}

\renewcommand{\emptyset}{\varnothing}
\newcommand{\KK}{\mathbf{K}}
\newcommand{\FF}{\mathbf{F}}
\newcommand{\NN}{\mathbb{N}} 
\newcommand{\PP}{\mathbb{P}} 
\newcommand{\QQ}{\mathbb{Q}} 
\newcommand{\RR}{\mathbb{R}} 

\newtheorem{theorem}{Theorem}[section]

\newtheorem{corollary}[theorem]{Corollary}
\newtheorem{proposition}[theorem]{Proposition}

\theoremstyle{definition}

\newtheorem{remark}[theorem]{Remark}

\begin{document}
\title{On arithmetic general theorems for polarized varieties}

\author{Nathan Grieve}
\address{Department of Mathematics, Michigan State University,
East Lansing, MI, 48824, USA}
\email{grievena@msu.edu}%

\begin{abstract} 
We apply Schmidt's Subspace Theorem to establish Arithmetic General Theorems for projective varieties over number and function fields.  Our first result extends an analogous result of M.~Ru and P.~Vojta.  One aspect to its proof makes use of a filtration construction which appears in work of Autissier.  Further, we consider work of M.~Ru and J.~T.-Y.~Wang, which pertains to an extension of K.~F. Roth's theorem for projective varieties in the sense of D.~McKinnon and M.~Roth.  Motivated by these works, we establish our second Arithmetic General Theorem, namely a form of Roth's theorem for exceptional divisors.  Finally, we observe that our results give, within the context of Fano varieties, a sufficient condition for validity of the main inequalities predicted by Vojta.  
\end{abstract}
\thanks{\emph{Mathematics Subject Classification (2010):} 14G05, 14G40, 11G50, 11J97, 14C20.}

\maketitle

\section{Introduction}

\np  The classical theory of Weierstrass points of algebraic curves highlights beautiful interplay between the geometry, combinatorics, analysis and arithmetic of a given compact Riemann surface and its Jacobian.  That story also suggests that divisorial filtrations of linear series on higher dimensional projective varieties are also of independent interest.  For example, they should have applications to the arithmetic, moduli and stability of projective varieties.  

\np  In \cite{Grieve:chow:approx}, we obtained results in this direction.  Indeed, motivated by \cite{McKinnon-Roth} and \cite{Grieve:Function:Fields}, we showed how a form of Roth's theorem, for polarized projective varieties, is related to important combinatorial and moduli theoretic invariants.  One feature of our \cite{Grieve:chow:approx}, is that we expound upon nontrivial connections between: (i)  Geometric Invariant Theory, especially the theory of Chow forms of polarized varieties; (ii) the theories of Okounkov bodies and graded linear series; and (iii) Diophantine Geometry.  Among other works, some more recent starting points for our \cite{Grieve:chow:approx} include \cite{Boucksom:Chen:2011}, \cite{Boucksom:Kuronya:Maclean:Szemberg:2015} and 
\cite{Lazarsfeld:Mustata:2009}.

\np  The interplay between Diophantine Geometry, Geometric Invariant Theory and Probability Theory is well known.  In this spirit, some representative works include \cite{Evertse:Ferretti:2002}, \cite{Faltings:Wustholz}, \cite{Ferretti:2000} and \cite{Ferretti:2003}.  At the same time, we have made  progress in our understanding of concepts within Toric Geometry, for example the theories of test configurations, K-stability and the Duistermaat-Heckman measures.  As two important works in this area, we mention \cite{Donaldson:2002} and \cite{Boucksom-Hisamoto-Jonsson:2016}.

\np  Returning to questions in arithmetic, an important  insight of P.~Corvaja and U.~Zannier was the fact that Schmidt's Subspace Theorem, combined with a detailed understanding of orders of vanishing of sections at rational points, could be used to prove the celebrated Siegel's theorem on $S$-integral points of affine curves, \cite{Corvaja:Zannier:2002}.  Both Schmidt's Subspace Theorem and the filtration construction of Corvaja-Zannier have been further refined and applied on a number of different occasions (\cite{Corvaja:Zannier:2004a}, \cite{Corvaja:Zannier:2004b}, \cite{Levin:2009}, \cite{Autissier:2011} and \cite{Levin:2014}).  

\np  Our main purpose here is to study these above mentioned topics within the context of work of Ru-Vojta, \cite{Ru:Vojta:2016}, which extends earlier work of Ru, \cite{Ru:2017}.  In doing so, we obtain results which concern the  arithmetic of filtered linear series on projective varieties.  Our results complement those of \cite{Ru:Vojta:2016} and \cite{Ru:2017}.  For example, they pertain to the nature of Weil functions evaluated at rational points outside of proper Zariski closed subsets.  In proving our main result, Theorem \ref{arithmetic:general:thm:volume:constant}, we study the main Arithmetic General Theorem of \cite{Ru:Vojta:2016}.  We generalize that result, see Theorem \ref{general:Arithmetic:General:Theorem}, and give applications, Corollary  \ref{Roth:exceptional:divisor:intro} and Corollary  \ref{vojta:inequalities}.  

\np  For more precise statements, an important aspect of \cite{Ru:Vojta:2016} and \cite{Ru:2017} is the use of Schmidt's Subspace Theorem to deduce Arithmetic General Theorems.  Those results apply to polarized varieties over number fields. Here we view these theorems from a broader perspective.  Specifically, we use the techniques of \cite{Ru:Vojta:2016}, which build on those of \cite{Levin:2009}, \cite{Autissier:2011} and \cite{Ru:2017}, to establish Theorem \ref{arithmetic:general:thm:volume:constant} below.  One feature of our results here is that they apply to polarized varieties over number and characteristic zero function fields.  Again, as in \cite{Ru:Vojta:2016}, one key point is an application of Schmidt's Subspace Theorem.  Although our proof of Theorem \ref{arithmetic:general:thm:volume:constant} is based on that of \cite{Ru:Vojta:2016}, here we offer a slight conceptual improvement.  Indeed, we express some aspects of the proof using expectations of discrete measures.  

\np  The use of such measures in the study of diophantine and arithmetic aspects of linear series on projective varieties is well established in the literature, \cite{Faltings:Wustholz}, \cite{Ferretti:2000}, \cite{Ferretti:2003}, \cite{Boucksom:Chen:2011}, \cite{Grieve:chow:approx}.  Similar kinds of expectations and discrete measures also play an important role in the K-stability of projective varieties, \cite{Boucksom-Hisamoto-Jonsson:2016}.  They also have important connections to measures of asymptotic growth of linear series, \cite{Boucksom:Kuronya:Maclean:Szemberg:2015}.  As it turns out, the study of those works was one impetus to the Arithmetic General Theorems that we establish here.   

\np  In a related direction, we establish a form of Roth's Theorem for exceptional divisors, see Corollary \ref{Roth:exceptional:divisor:intro}.  That result is motivated by the main theorem obtained by Ru-Wang, \cite{Ru:Wang:2016}, the main results of McKinnon-Roth, \cite{McKinnon-Roth}, \cite{McKinnon-Roth-Louiville}, as well as our related works, \cite{Grieve:Function:Fields}, \cite{Grieve:chow:approx}.  We also deduce further consequences, Corollary  \ref{vojta:inequalities} and Theorem \ref{vojta:inequalities:thm}, that pertain to Vojta's main conjecture.    As one other complementary result, a form of Schmidt's Subspace Theorem, which involves Weil functions and Seshadri constants for subvarieties, has been obtained by Heier-Levin, \cite{Heier:Levin:2017}.  

\np  To state Theorem \ref{arithmetic:general:thm:volume:constant}, let $\KK$ be a number field or a characteristic zero function field with fixed algebraic closure $\overline{\KK}$.  Let $\FF / \KK$ be a finite extension, $\KK \subseteq \FF \subseteq \overline{\KK}$, let $S$ be a fixed finite set of places of $\KK$ and extend each $v \in S$ to a place $v$ of $\FF$.
Suppose that $X$ is a geometrically irreducible, geometrically normal projective variety over $\KK$.  Fix nonzero effective Cartier divisors $D_1,\dots, D_q$ on $X$, defined over $\FF$, and fix Weil functions $\lambda_{D_i,v}(\cdot)$ for each $D_i$ and each $v \in S$. 
Our conventions are such that each of the Weil functions $\lambda_{D_i,v}(\cdot)$ are normalized relative to the base field $\KK$.  

\np  Let $L$ be a \emph{big line bundle} on $X$, defined over $\KK$.  We consider the \emph{asymptotic volume constants} of $L$ along each of the Cartier divisors $D_i$:
\begin{equation}\label{Asym:Vol:Costants}
\beta(L,D_i) := \int_0^{\infty} \frac{\Vol(L-tD_i)}{\Vol(L)} \mathrm{d} t.
\end{equation}
Put 
$$D = D_1 + \hdots + D_q$$
 and, for each $v \in S$, choose $\lambda_{D,v}(\cdot)$, a Weil function for $D$ with respect to $v$.  The \emph{proximity function} of $D$ with respect to $S$ is then:
$$
m_S(\cdot,D) = \sum_{v \in S} \lambda_{D,v}(\cdot).
$$

\np  Our main result gives a unified picture of the related works, \cite{McKinnon-Roth}, \cite{Grieve:Function:Fields}, \cite{Ru:Vojta:2016}, \cite{Ru:Wang:2016} and \cite{Grieve:chow:approx}.  It also generalizes the Arithmetic General Theorem of \cite{Ru:Vojta:2016}.  For example, Theorem \ref{arithmetic:general:thm:volume:constant} below also applies to the case that $\KK$ is a characteristic zero function field.  It also allows for each of the Cartier divisors $D_i$ to be defined over $\FF / \KK$ a finite extension of the base field.

\begin{theorem}\label{arithmetic:general:thm:volume:constant}
Let $L$ be a big line bundle on $X$ and suppose that the Cartier divisors $D_1,\dots,D_q$ intersect properly.  Let $\epsilon > 0$.  There exists constants $a_\epsilon, b_\epsilon > 0$ with the property that either:
$$ h_L(x) \leq a_\epsilon;$$
or 
$$
m_S(x,D) \leq \left( \max_{1 \leq j \leq q} \beta(L,D_j)^{-1} + \epsilon \right) h_L(x) + b_\epsilon
$$
for all $\KK$-rational points $x \in X(\KK)$ outside of some proper Zariski closed subset $Z \subsetneq X$.
\end{theorem}

\np  By analogy with \cite[Section 3]{Boucksom-Hisamoto-Jonsson:2016}, the constants $\beta(L,D_i)$ can be identified with the limit expectation of the \emph{Duistermaat-Heckman measures}.  Further, when $L$ is assumed to be very ample, $\beta(L,D_i)$ is identified with the normalized Chow weight of $L$ along $D_i$, \cite{Grieve:chow:approx}.  We discuss these matters in Section \ref{jet:amplitude}.

\np  Theorem \ref{arithmetic:general:thm:volume:constant} implies a form of Roth's theorem for exceptional divisors.  Let 
$$\pi\colon \widetilde{X} \rightarrow X_\FF$$ 
be the blowing-up of $X$ along a proper subscheme $Y \subsetneq X$, defined over $\FF$, with exceptional divisor $E$.   The  \emph{asymptotic volume constant} takes the form:
$$ 
\beta(\pi^*L, E) = \beta_E(L) = \mathbb{E}(\nu) = \lim_{m \to \infty} \mathbb{E}(\nu_m).
$$
Here $\mathbb{E}(\nu_m)$ is the expectation of the discrete measures $\nu_m$ determined by  the filtration of $\H^0(X_\FF, mL_\FF)$ which is induced by $\mathcal{F}^\bullet R(L_\FF)$, the filtration of the section ring $R(L_\FF)$ which is determined by $E$.
Again, by analogy with \cite[Section 3]{Boucksom-Hisamoto-Jonsson:2016}, the limit expectation $\mathbb{E}(\nu)$ can also be identified with the limit expectation of the \emph{Duistermaat-Heckman measures}, compare also with \cite[Section 1]{Boucksom:Chen:2011} and \cite[Section 1]{Boucksom:Kuronya:Maclean:Szemberg:2015}.

\begin{corollary}\label{Roth:exceptional:divisor:intro}
If $\epsilon > 0$, then there exists positive constants $a_\epsilon$ and $b_\epsilon$ so that either:
$$h_L(x) \leq a_\epsilon; $$
or
$$
\sum_{v \in S} 
\lambda_{E,v}(x) \leq \left( \frac{1}{\beta_E(L)} + \epsilon \right) h_L(x) + b_{\epsilon}
$$
for all $\KK$-rational points $x \in X(\KK)$ outside of some proper subvariety $Z \subsetneq X$.
\end{corollary}

We prove Corollary  \ref{Roth:exceptional:divisor:intro} in Section \ref{arithmetic:application:proofs}.   

\np  One additional implication of Theorem \ref{arithmetic:general:thm:volume:constant} and  Theorem \ref{general:Arithmetic:General:Theorem} applies to \emph{Fano varieties}.  For example, we obtain a sufficient condition for validity of the inequalities predicted by Vojta's main conjecture, \cite{Vojta:1987}.  This observation, Corollary \ref{vojta:inequalities} below, is a natural extension to \cite[Theorem 10.1]{McKinnon-Roth}.  It has also been noted by Ru and Vojta, \cite[Corollary 1.12]{Ru:Vojta:2016}.  Here, we view it as a consequence of a more general statement; see Theorem \ref{vojta:inequalities:thm}.

\begin{corollary}
\label{vojta:inequalities}
Suppose $X$ that is nonsingular, that $-\K_X$ is ample and that  
$$D = D_1 + \hdots + D_q$$ 
is a simple normal crossings divisor on $X$.  In this context, if 
$$\beta(- \K_X, D_i) \geq 1,$$ 
for all $i$, then Vojta's inequalities hold true.  Precisely, if $M$ is a big divisor on $X$, then for all $\epsilon > 0$, there exists constants $a_\epsilon, b_\epsilon > 0$, so that either
$$ 
h_{- \K_X}(x) \leq a_\epsilon
$$
or
$$
m_{S}(x,D) + h_{\K_X}(x) \leq \epsilon h_M(x) + b_{\epsilon}
$$
for all $\KK$-rational points $x \in X(\KK)$ outside of some proper subvariety $Z \subsetneq X$.
\end{corollary}

\np  We prove Corollary  \ref{vojta:inequalities} in Section \ref{arithmetic:application:proofs}.
When $-\K_X$ is very ample, then, as is a consequence of \cite[Theorem 5.2]{Grieve:chow:approx}, the condition that $\beta(- \K_X, D_i) \geq 1$ can be expressed in terms of a normalized Chow weight for the anti-canonical embedding.   From this point of view, Vojta's inequalities are implied by a sort of Chow stability condition. 

\np  To put matters into perspective, Corollary  \ref{vojta:inequalities} raises the question as to the extent to which its hypothesis can be satisfied in nontrivial situations.  To that end,  Corollary  \ref{vojta:inequalities} is a special case of Theorem \ref{vojta:inequalities:thm}.  That result can be applied in a variety of contexts including the examples mentioned in \cite[Section 10]{McKinnon-Roth}.  Studying the hypothesis of Corollary  \ref{vojta:inequalities} and Theorem \ref{vojta:inequalities:thm} in detail is a topic, of independent interest, which we do not pursue here.

\np  In proving Theorem \ref{arithmetic:general:thm:volume:constant}, we establish a form of Schmidt's Subspace Theorem for effective linear series on projective varieties (Proposition \ref{Schmidt:Subspace:Thm}).  In Theorem \ref{general:Arithmetic:General:Theorem},  we also extend the main Arithmetic General Theorem of Ru and Vojta, \cite{Ru:Vojta:2016}.   Finally, we use results from our \cite{Grieve:chow:approx}, which build on results and ideas from \cite{Boucksom:Chen:2011}, \cite{Boucksom:Kuronya:Maclean:Szemberg:2015} and \cite{Boucksom-Hisamoto-Jonsson:2016}, to interpret the asymptotic volume constant in terms of orders of vanishing along subvarieties.  This theme appears in various contexts within Diophantine Approximation, for example \cite{Autissier:2011}, \cite{Ru:Vojta:2016} and \cite{Ru:Wang:2016}.   We make those connections precise in Section \ref{jet:amplitude}.

\subsection{Notations and conventions}  Throughout this article, $\KK$ denotes a number field or a characteristic zero function field. Further, $\FF / \KK$ denotes a fixed finite extension of $\KK$ contained in $\overline{\KK}$ a fixed algebraic closure of $\KK$.  All absolute values, height functions and Weil functions are normalized relative to $\KK$.  

By a \emph{variety}, we mean a reduced projective scheme over a given base field. If $L$ is a line bundle on a geometrically irreducible variety $X$, defined over $\KK$, then $L_\FF$ denotes its pullback to 
$$X_\FF := X \times_{\spec(\KK)}  \spec(\FF).$$  

Using terminology that is consistent with \cite{Ru:Vojta:2016}, we say that a collection of Cartier divisors $D_1,\dots,D_q$ on $X$, defined over $\FF$, intersect properly if for all nonempty subsets 
$$\emptyset \not = I \subseteq \{1,\dots, q \}$$ 
and each $x \in \bigcap_{i \in I} D_i$, the collection of local equations for the $D_i$ near $x$ form a regular sequence in the local ring $\Osh_{X,x}$.  

We also denote the \emph{greatest lower bound} of $D_1,\dots,D_q$ by $\bigwedge_{i \in I} D_i$ and their \emph{least upper bound} by $\bigvee_{i \in I} D_i$.  In particular, in line with \cite{Ru:Vojta:2016}, we may realize these (birational) divisors as Cartier divisors on some normal proper model of $X$.

Finally, when no confusion is likely, we often times use additive (as opposed to multiplicative) notation to denote tensor products of coherent sheaves.

\subsection{Acknowledgements}  
I thank Aaron Levin and Steven Lu for helpful discussions and encouragement.  I also thank Julie Wang for conversations and comments on related topics and for sharing with me a preliminary version of her joint work \cite{Ru:Wang:2016}.  Finally, I thank anonymous referees for their remarks and am especially grateful to Min Ru for helpful comments and for sharing with me the final version of his joint work \cite{Ru:Vojta:2016}.  This work was conducted while I was a postdoctoral fellow at Michigan State University.  

\section{Preliminaries}\label{Preliminaries:Weil:height}

In this section, we make precise concepts and conventions which are required for what follows.  We also establish Proposition \ref{Schmidt:Subspace:Thm} which we use, in Section \ref{Arithmetic:General:Theorem:Proof}, to establish Theorem \ref{general:Arithmetic:General:Theorem}.  Our approach to absolute values and Weil functions is similar to \cite[Section 2]{Grieve:Function:Fields} and \cite[Section 1]{Bombieri:Gubler}.

\subsection{Absolute values and the product formula}  Let $\KK$ be a number field or a characteristic zero function field with set of places $M_\KK$.  For each $v \in M_\KK$, there exists a representative $|\cdot|_v$ so that the collection of such absolute values $|\cdot|_v$, $v \in M_\KK$, satisfy the product rule:
$$\prod_{v \in M_\KK} |\alpha|_v = 1, $$
for each $\alpha \in \KK^\times$.  Our choice of representatives $|\cdot|_v$, $v \in M_\KK$, are made precise in the following way.  When $\KK = \QQ$, if $p \in M_\QQ$ and $p = \infty$, then $|\cdot|_p$ is the ordinary absolute value on $\QQ$.  If $p \in M_\QQ$ is a prime number, then $|\cdot|_p$ is the usual $p$-adic absolute value on $\QQ$, normalized by the condition that $|p|_p = p^{-1}$.  For a general number field $\KK$, if $\mathfrak{p} \in M_\KK$, then $\mathfrak{p} | p$ for some $p \in M_\QQ$ and we put:
$$ 
|\cdot|_{\mathfrak{p}} := \left|\N_{\KK_{\mathfrak{p}} / \QQ_p}(\cdot)\right|_p^{1 / [\KK : \QQ]}.
$$
Here $\N_{\KK_{\mathfrak{p} / \QQ_p}}(\cdot) : \KK_{\mathfrak{p}} \rightarrow \QQ_p$ is the norm from $\KK_{\mathfrak{p}}$ to $\QQ_p$.  

When $\KK$ is a function field, $\KK = \kk(Y)$ for $(Y,\mathcal{L})$ a polarized variety over an algebraically closed base field $\kk$, $\cchar (\kk) = 0$, and $Y$ is irreducible and nonsingular in codimension one.  The local ring $\Osh_{Y,\mathfrak{p}}$ of a prime divisor $\mathfrak{p} \subseteq Y$ is a discrete valuation ring with discrete valuation $\operatorname{ord}_{\mathfrak{p}}(\cdot)$; put:
$$ 
|\cdot|_{\mathfrak{p}} := \mathrm{e}^{- \operatorname{ord}_{\mathfrak{p}}(\cdot) \deg_{\mathcal{L}}(\mathfrak{p})}.
$$
In this context, we also write $M_\KK = M_{(Y, \mathcal{L})}$ to indicate dependence on the polarized variety $(Y,\mathcal{L})$.  

In general, given a finite extension $\FF / \KK$, if $w \in M_\FF$ is a place of $\FF$ lying over a place $v \in M_\KK$, then we put: 
$$
|| \cdot ||_w := \left| \N_{\FF_w / \KK_v}(\cdot) \right|_v.
$$
Further, let:
$$ 
\left | \cdot \right |_{w,\KK} := \left| \N_{\FF_w / \KK_v} (\cdot) \right|^{\frac{1}{[\FF_w : \KK_v ]}}_v = || \cdot ||_w^{\frac{1}{[\FF_w : \KK_v ] }}.
$$
Then $\left | \cdot \right |_{w,\KK}$ is an absolute value on $\FF$ which represents $w$ and extends $|\cdot|_v$.  We often identify $v$ and $w$ and write:
\begin{equation}\label{absolute:value:extend}
\left | \cdot \right |_{v,\KK} := 
\left| \cdot \right |_{w, \KK} 
=
\left| \N_{\FF_w / \KK_v} (\cdot) \right|^{\frac{1}{[\FF_w : \KK_v ]}}_v 
= 
|| \cdot ||_v^{\frac{1}{[\FF_w : \KK_v ] }}.
\end{equation}

Finally, in case that $\KK = \kk(Y)$ is a function field, then $\FF = \kk(Y')$ for $Y' \rightarrow Y$ the normalization of $Y$ in $\FF$.  In this context, $M_\FF = M_{(Y',\mathcal{L}')}$, for $\mathcal{L}'$ the pullback of $\mathcal{L}$ to $Y'$.

\subsection{Weil and Proximity functions for Cartier divisors}  
Let $X$ be a geometrically irreducible projective variety over $\KK$.  We consider local Weil functions determined by meromorphic sections of $P$, a line bundle on $X$ defined over $\FF$.  Fix global sections $s_0,\dots,s_n$ and $t_0,\dots,t_m$ of globally generated line bundles $M$ and $N$ on $X$, defined over $\FF$, and so that 
$$P \simeq M \otimes N^{-1}.$$  
Fix $s$ a meromorphic section of $P$.  As in \cite{Bombieri:Gubler}, we say that
$$ 
\mathcal{D} = \mathcal{D}(s) = (s; M,\mathbf{s};N,\mathbf{t})
$$
is a \emph{presentation} of $P$ with respect to $s$ (defined over $\FF$).  

We define the \emph{local Weil function of $P$}, with respect to $\mathcal{D}$ and a fixed place $v \in M_\KK$, by:
\begin{equation}\label{local:weil:defn} 
\lambda_{\mathcal{D}}(y,v) :=  
\max_k \min_\ell \log \left | \frac{s_k}{t_\ell s} (y) \right|_{v, \KK} 
= \max_k \min_\ell \log \left| \left | \frac{s_k}{t_\ell s} (y) \right | \right|_v^{\frac{1}{[\FF_w : \KK_v] }}.
\end{equation}
This function has domain $X(\KK) \setminus \operatorname{Supp}(s)(\KK)$.  In \eqref{local:weil:defn}, the place $v$ is extended to $\FF$ and the absolute value $|\cdot|_{v,\KK}$ is defined as in \eqref{absolute:value:extend}. 
Some basic properties of local Weil functions can be deduced along the lines of \cite[Section 2.2]{Bombieri:Gubler} and \cite[Chapter 10]{Lang:Diophantine}.  

At times, we identify a meromorphic section $s$ of $P$ with the Cartier divisor 
$$D = \operatorname{div}(s)$$ 
on $X_\FF$ that it determines.   In this context, we also say that $D$ is a \emph{Cartier divisor} on $X$, \emph{defined} over $\FF$.  We also use the notations $\lambda_{D,v}(\cdot)$, $\lambda_{s,v}(\cdot)$ and $\lambda_{\mathcal{D}(s)}(\cdot,v)$ to denote the local Weil function $\lambda_{\mathcal{D}}(\cdot,v)$.

We define the \emph{global Weil function} for $\mathcal{D}$, normalized relative to $\KK$, by:
$$ 
\lambda_{\mathcal{D}}(\cdot) = \sum_{v \in M_\KK} \lambda_{\mathcal{D}}(\cdot;v).
$$
This function has domain the set of $\KK$-points $x \in X(\KK)$ outside of the support of $D = \operatorname{div}(s)$.  

The \emph{proximity function} of $D$ with respect to a finite set $S \subsetneq M_\KK$ of places of $\KK$ also has domain  $X(\KK) \setminus \operatorname{Supp}(D) (\KK)$.  It is defined by:
$$ 
m_S(\cdot,D) := \sum_{v \in S} \lambda_{D,v}(\cdot).
$$

\subsection{Weil functions for subschemes}
Similar to \cite{Silverman:1987}, via Weil functions of the exceptional divisor $E$ of 
$$\pi \colon \widetilde{X} \rightarrow X_\FF,$$ 
the blowing-up of $X$ along a proper subscheme $Y \subsetneq X$, defined over $\FF$, we have concepts of $\lambda_{Y,v}(\cdot)$, the Weil function of $Y$ with respect to places $v \in S$.  More precisely, we identify $E$ with a line bundle on $\widetilde{X}$ and defined over $\FF$.  We then define the Weil function for $Y$ with respect to places $v$ of $\KK$ (extended to $\FF$) via:
$$\lambda_{Y,v}(\cdot) := \lambda_{E,v}(\cdot);$$ 
such functions have domain the set of $\KK$-points of $X$ outside of the support of $Y$.

\subsection{Logarithmic height functions}  Let $h_L(\cdot)$ denote the logarithmic height of a line bundle $L$ on $X$, defined over $\KK$.  This function is normalized relative to the base field $\KK$.  When $L$ is very ample, $h_L(\cdot)$ is obtained by pulling back the standard logarithmic height function on projective space $\PP^n_\KK$: 
$$h_{\Osh_{\PP^n}(1)}(\mathbf{x}) = \sum_{v \in M_\KK} \max_j \log | x_j |_v,$$
for $\mathbf{x} = (x_0,\dots,x_n) \in \PP^n(\KK)$ and $n = h^0(X,L) - 1$.  In general, $h_L(\cdot)$ is obtained by first expressing $L$ as a difference of very ample line bundles $L \simeq M - N$.  Then $h_L(\cdot)$ is defined, up to equivalence, as the difference of $h_M(\cdot)$ and $h_N(\cdot)$.  

\subsection{Schmidt's Subspace Theorem}  For later use, see Theorem \ref{general:Arithmetic:General:Theorem}, we require an extension of Schmidt's Subspace Theorem.

\begin{proposition}[Compare with {\cite[Theorem 2.7]{Ru:Vojta:2016}}]
\label{Schmidt:Subspace:Thm}  Let $X$ be a geometrically irreducible projective variety over $\KK$ and $L$ an effective line bundle on $X$, defined over $\KK$.  Let 
$$0 \not = V \subseteq \H^0(X,L)$$
be an effective linear system, put $n := \dim |V|$, and for each place $v \in S$, fix  a collection of sections:
$$ \sigma_{1,v},\dots,\sigma_{q,v} \in V_\FF := \FF \otimes_\KK V \subseteq \H^0(X_\FF, L_\FF),$$ 
with Weil functions $\lambda_{\sigma_{i,v},v}(\cdot)$.
Let $\epsilon > 0$.  Then there exists positive constants $a_\epsilon$ and $b_\epsilon$ together with a proper Zariski closed subset $Z \subsetneq X$ so that for all $x \in X(\KK) \setminus Z(\KK)$, either
$$h_L(x) \leq a_\epsilon;$$ or 
$$\max_J \sum_{j \in J}  \sum_{v \in S} \lambda_{\sigma_{j,v},v}(x) \leq (\epsilon + n + 1) h_{L}(x) + b_\epsilon.
$$
Here the maximum is taken over all subsets $J \subseteq \{1,\dots,q\}$ for which the sections $\sigma_{j,v}$, with $j \in J$, are linearly independent.
\end{proposition}

\begin{remark}\label{Schmidt:Subspace:Thm:Rmk}
In Proposition \ref{Schmidt:Subspace:Thm}, the subvariety $Z$ is essentially determined by the $\KK$-span of a finite collection of hyperplane sections.
\end{remark}

\begin{proof}[Proof of Proposition \ref{Schmidt:Subspace:Thm}]
Fix $v \in S$.  Each basis 
$$
\sigma_{0,v},\dots,\sigma_{n,v} \in V_\FF \subseteq \H^0(X_\FF, L_\FF)
$$ 
determines a rational map
$
\phi_{V_\FF,v} \colon X_\FF \dashrightarrow \PP^n_\FF.
$
By elimination of indeterminacy of rational maps, we obtain a commutative diagram of $\FF$-schemes:
$$
\begin{tikzcd}
\pi^{-1}(U_\FF) \arrow[hook]{r} \arrow[swap]{d}{\wr} &\widetilde{X}_\FF \arrow[swap]{d}{\pi} \arrow{dr}{\widetilde{\phi}_{\widetilde{V}_\FF}} \\
U_\FF \arrow[hook]{r} \arrow[swap, bend right]{rr}{\phi \mid_{U_\FF}} & X_\FF \arrow[dashed]{r}  &\PP^n_\FF. 
\end{tikzcd}
$$
In particular, 
$$\widetilde{X}_{\FF} = \widetilde{X} \times_{\spec(\KK)} \spec(\FF),$$ 
for some $\KK$-variety $\widetilde{X}$, independent of $v \in S$.   The morphism $\pi$ is projective and there exists effective line bundles $M$ and $B$ on $\widetilde{X}$, defined over $\KK$, with the properties that:
$$
M_\FF \simeq \widetilde{\phi}^*_{\widetilde{V}_\FF} \Osh_{\PP^n_\FF}(1) = \pi^*L_\FF - B_\FF
$$
and
$$
\pi^*\sigma_{iv} \in \H^0(\widetilde{X},M_\FF)\text{,}
$$
for each $i = 0,\dots, n$ and each $v \in S$.  Consider the Weil functions $\lambda_{\pi^*\sigma_{i,v},v}(\cdot)$ for $M$ with respect to the global sections $\pi^* \sigma_{i,v}$, the pullback of the Weil functions $\lambda_{\sigma_{i,v},v}(\cdot)$ for $L$ and the relation:
\begin{equation}\label{P:eqn5}
\pi^* L_\FF \simeq M_\FF + B_\FF.
\end{equation}
It follows from \eqref{P:eqn5} that the logarithmic heights of the line bundles $\pi^* L$, $M$ and $B$ are related by:
\begin{equation}\label{P:eqn6}
h_{\pi^*L}(\cdot) = h_M(\cdot) + h_B(\cdot) + \mathrm{O}(1).
\end{equation}
Furthermore, the local Weil functions are related via: \begin{equation}\label{P:eqn8}
\pi^* \lambda_{\sigma_i,v}(\cdot) = \lambda_{\pi^*\sigma_{i,v},v}(\cdot) + \lambda_{B,v}(\cdot) + \mathrm{O}(1).
\end{equation}
Here, 
$$\lambda_{B,v}(\cdot) = \lambda_B(\cdot,v)$$ 
is a fixed Weil function for $B$, depending on the choice of a meromorphic section.  The conclusion desired by Proposition \ref{Schmidt:Subspace:Thm} can now be deduced by combining the equations \eqref{P:eqn6} and \eqref{P:eqn8} together with the version of Schmidt's Subspace Theorem as described for instance in \cite[Theorem 7.2.2]{Bombieri:Gubler} or \cite[Theorem 5.2]{Grieve:Function:Fields}.
\end{proof}

\section{Amplitude of Jets and the Filtration construction of Autissier}\label{jet:amplitude}

Here we discuss complexity of filtered linear series.  The main point is to give a unified interpretation of concepts of growth of filtered series which appear in \cite{Autissier:2009}, \cite{Autissier:2011}, \cite{McKinnon-Roth} and \cite{Ru:Vojta:2016}.

\np  Let $L$ be a big line bundle on $X$, an irreducible projective variety over $\kk$, an algebraically closed field of characteristic zero, with section ring
$$ R(L) = R(X,L) = \bigoplus_{m \geq 0} \H^0(X, mL).$$
The \emph{vanishing numbers} of a decreasing, multiplicative $\RR$-filtration 
$$\mathcal{F}^\bullet = \mathcal{F}^\bullet R(L)$$ 
of $R(L)$ are defined, for each $m \geq 0$, by the condition that:
$$ 
a_j(mL) = \inf \left\{ t \in \RR : \operatorname{codim} \mathcal{F}^t \H^0(X, mL) \geq j + 1 \right\}.
$$
Such filtrations determine discrete measures on the real line:
$$ \nu_{m,\mathcal{F}^\bullet} = \nu_m := \frac{1}{h^0(X,mL)} \sum_j \delta_{m^{-1} a_j(mL)},$$
for each $m \geq 0$.

\np  Let $E$ be an effective line bundle on $X$.  Put:
$$
\alpha(L,E) := \frac{\sum_{\ell \geq 1} h^0(X, L - \ell E)}{h^0(X,L)}
$$
and:
\begin{equation}\label{jet:eqn3}
\beta(L,E) = \liminf_{m \to \infty} m^{-1} \alpha(mL,E) = \liminf_{m \to \infty} \frac{\sum_{\ell \geq 1} h^0(X, mL - \ell E)}{m h^0(X, mL)}.
\end{equation}
For later use, set:
$$
\gamma(L,E) = \beta(L,E)^{-1},
$$
 \cite{Ru:Vojta:2016} or \cite{Autissier:2011},
 and
 $$
 \beta_E(L) := \int_0^\infty \frac{\Vol(L - t E)}{\Vol(L)} \mathrm{d}t,
 $$
 \cite{McKinnon-Roth}.

\np  Suppose that $X$ is normal and that $L$ is a big line bundle on $X$.  We establish equivalence of the quantities $\beta(L,E)$ and $\beta_E(L)$.  

\begin{proposition}\label{jet:prop:1}
If $X$ is normal and if $L$ is a big line bundle on $X$, then:
\begin{equation}\label{jet:eqn4}
\beta(L,E) = \beta_E(L).
\end{equation}
\end{proposition}

\begin{proof}
Let 
$$\mathcal{F}^\bullet = \mathcal{F}^\bullet R(L)$$ 
be the filtration of $R(L)$ induced by orders of vanishing along $E$, with vanishing numbers $a_j(mL)$ and discrete measures $\nu_m$.  Put:
$$
h_m(t) := \frac{h^0(X, m L - m t E ) }{h^0(X, mL)}
$$
and
$$
\mathbb{E}(\nu_m) := \int_0^\infty t \cdot \mathrm{d}(\nu_m).
 $$ 
As in \cite[Theorem 1.11]{Boucksom:Chen:2011}, we can write:
$$
\mathbb{E}(\nu_m) = \int_0^{\infty} t \cdot \mathrm{d}(\nu_m) = - \int_0^{\infty} t \cdot h_m'(t) \mathrm{d} t = \int_0^{\infty} h_m(t) \mathrm{d}t.
$$
Further, we know from \cite[Theorem 1.11]{Boucksom:Chen:2011} that:
\begin{equation}\label{jet:eqn8}
\beta_E(L) = \lim_{m \to \infty} \mathbb{E}(\nu_m).
\end{equation}
The functions $h_m(t)$ are also constant on intervals of width $m$.  Thus:
\begin{equation}\label{jet:eqn9}
\int_0^{\infty} h_m(t) \mathrm{d}t = \frac{1}{m h^0(X, mL)} \sum_{\ell = 0}^{ \infty } h^0(X, mL - \ell E).
\end{equation}
Rewriting \eqref{jet:eqn9}, we get:
\begin{equation}\label{jet:eqn10}
\int_0^{\infty} h_m(t) \mathrm{d}t = m^{-1} + \frac{m^{-1}}{h^0(X,mL)} \sum_{ \ell = 1}^{ \infty } h^0(X, mL - \ell E).
\end{equation}
Because of \eqref{jet:eqn8} and dominated convergence,  \eqref{jet:eqn10} implies, as $m \to \infty$:
\begin{equation}\label{jet:eqn11}
\beta_E(L) = \mathbb{E}(\nu) = \lim_{m \to \infty } \frac{\sum_{ \ell = 1}^{ \infty  } h^0(X, mL - \ell E) }{m h^0(X, mL)}.
\end{equation}
\end{proof}

\np  When $L$ is very ample, there is also an interpretation as a \emph{normalized Chow weight}.

\begin{corollary}\label{jet:cor:2}
If $X$ is normal and $L$ is very ample, then:
\begin{equation}\label{jet:eqn5}
\beta(L,E) = \beta_E(L) = \frac{e_X(\mathbf{c})}{(\dim X + 1)(\deg_L X)} = \lim_{m\to \infty} \frac{\sum_{\ell \geq 1} h^0(X, mL - \ell E)}{m h^0(X, mL)}.
\end{equation}
Here $e_X(\mathbf{c})$ is the Chow weight of $X$ in $\PP^n_\kk$ with respect to 
$$\mathbf{c} = (a_0(L),\dots,a_n(L)),$$ 
the weight vector of $L$ along $E$, and the embedding of $X$ into $\PP^n_\kk$ is induced by a basis of $\H^0(X,L)$ which is compatible with the filtration given by orders of vanishing along $E$.
\end{corollary}

\begin{proof}
We combine \eqref{jet:eqn3}, \eqref{jet:eqn4} and results from \cite{Grieve:chow:approx}.  For instance, a key point, \cite[Proof of Proposition 4.1]{Grieve:chow:approx}, is:
\begin{equation}\label{jet:eqn12}
\mathbb{E}(\nu_m) = \frac{s(m,\mathbf{c})}{m h^0(X, mL)}.
\end{equation}
In \eqref{jet:eqn12}, $s(m,\mathbf{c})$ denotes the \emph{$m$th Hilbert weight} of $X$ in $\PP^n_\kk$ with respect to the weight vector 
$$
\mathbf{c} =(a_0(L),\dots,a_n(L)).
$$
\end{proof}

\np  We now indicate how the filtration construction of Corvaja-Zannier, \cite{Corvaja:Zannier:2002}, Levin, \cite{Levin:2009}, Autissier, \cite{Autissier:2011}, and Ru-Vojta, \cite{Ru:Vojta:2016}, relates to Proposition \ref{jet:prop:1} and Corollary  \ref{jet:cor:2}.  Fix effective Cartier divisors $D_1,\dots, D_q$ on $X$.  We assume that these divisors intersect properly and put:
$$
\Sigma := \left\{\sigma \subseteq \{1,\dots,q \} : \bigcap_{j \in \sigma} \operatorname{Supp} D_j \not = \emptyset \right \}. 
$$  
Fix $\sigma \in \Sigma$ and, for each integer $b \geq 0$, let: 
$$ 
\Delta_\sigma = \Delta_{\sigma}(b) := \left \{\mathbf{a} = (a_i) \in \NN^{\# \sigma} : \sum_{i \in \sigma} a_i = b \right \}.
$$
Note that the set $\Delta_\sigma$ is finite.

\np  For each $t \in \RR_{\geq 0}$ and each $\mathbf{a} \in \Delta_{\sigma}$, define the ideal sheaf $\Ish(t;\mathbf{a};\sigma)$ by:
$$ 
\Ish(t;\mathbf{a};\sigma) = \sum_{
\substack{
\mathbf{b} \in \NN^{\# \sigma} \\
\sum_{i \in \sigma} a_i b_i \geq b t}
}
\Osh_X \left(-\sum_{i \in \sigma} b_i D_i \right).
$$
Note also that the ideal sheaf $\Ish(t;\mathbf{a};\sigma)$ is generated by finitely many such $\mathbf{b}$.

\np  In this way, we obtain a 
decreasing, multiplicative $\RR$-filtration of $R(L)$.  
Such filtrations have the form:
$$
\mathcal{F}^t(m;\sigma;\mathbf{a}) = \H^0\left(X,mL \otimes \Ish(t;\mathbf{a};\sigma) \right) \subseteq \H^0(X,mL),
$$
for $m \geq 0$.
If 
$$0 \not = s \in \H^0(X,mL),$$ 
then we let 
$$
\mu_{\mathbf{a}}(s) = \mu(s) = \sup \{ t \in \RR_{\geq 0} : s \in \mathcal{F}^t(m;\sigma;\mathbf{a}) \} = \max \{ t \in \RR_{\geq 0} : s \in \mathcal{F}^t(m;\sigma;\mathbf{a}) \}.
$$
In what follows, we denote by 
$$\mathcal{B}_{\sigma;\mathbf{a},m} = \mathcal{B}_{\sigma;\mathbf{a}}$$ 
a basis for $\H^0(X,mL)$ which is adapted to the filtration $\mathcal{F}^\bullet(m;\sigma;\mathbf{a})$.

\np  Fix an integer $m \geq 0$ and let 
$$a_{\min}(m;\sigma;\mathbf{a}) = a_0(m;\sigma,\mathbf{a}) \leq \dots \leq a_{n_m}(m;\sigma;\mathbf{a}) = a_{\max}(m;\sigma;\mathbf{a})
$$ 
denote the vanishing numbers of the filtrations 
$$\mathcal{F}^\bullet(m;\sigma;\mathbf{a}) = (\mathcal{F}^t(m;\sigma;\mathbf{a}))_{t \in \RR_{\geq 0}}.$$  
The discrete measures on $\RR$ that they determine are then:
$$ 
\nu(m;\sigma;\mathbf{a}) = \frac{1}{h^0(X,mL)} \sum_j \delta_{m^{-1} a_j(m;\sigma;\mathbf{a})};
$$
let $\mathbb{E}(m;\sigma;\mathbf{a})$ denote  the expectations of such measures:
$$
\mathbb{E}(m;\sigma;\mathbf{a}) := \int_0^\infty t \cdot \nu(m;\sigma;\mathbf{a}) \mathrm{d}t. 
$$

\begin{proposition}\label{ARV-filt-prop2}  Let $L$ be a big line bundle on $X$.  The expectations $\mathbb{E}(m;\sigma;\mathbf{a})$ have the properties that:
$$
\mathbb{E}(m;\sigma;\mathbf{a}) 
\geq \frac{1}{m h^0(X,mL)} \min_{1 \leq i \leq q} \left( \sum_{\ell \geq 1} h^0(X,mL - \ell D_i) \right) \text{.}
$$
\end{proposition}
\begin{proof}
This is a reformulation of \cite[Proposition 6.7]{Ru:Vojta:2016}.  In more detail, since the basis $\mathcal{B}_{\sigma;\mathbf{a};m}$ is adapted to the filtration, it follows that:
\begin{equation}\label{jet:prop:eqn:2}
\mathbb{E}(m;\sigma;\mathbf{a}) = \frac{1}{m h^0(X,mL)} \sum_{s \in \mathcal{B}_{\sigma;\mathbf{a};m}} \mu_{\mathbf{a}}(s).
\end{equation}
On the other hand, by \cite[Proposition 6.7]{Ru:Vojta:2016}, we have:
\begin{equation}\label{jet:prop:eqn3}
\frac{1}{h^0(X,mL)} \sum_{s \in \mathcal{B}_{\sigma;\mathbf{a};m}} \mu_{\mathbf{a}}(s) \geq \min_i \left( \frac{1}{h^0(X,mL)} \sum_{\ell \geq 1} h^0(X,mL - \ell D_i) \right).
\end{equation}
Thus, combining, \eqref{jet:prop:eqn:2} and \eqref{jet:prop:eqn3}, we obtain:
$$
\mathbb{E}(m;\sigma;\mathbf{a}) \geq \min_i \left( \frac{1}{m h^0(X,mL)} \sum_{\ell \geq 1} h^0(X, mL - \ell D_i) \right).
$$
\end{proof}

\np  We mention that the righthand side of Proposition \ref{ARV-filt-prop2} can be interpreted in terms of the \emph{Hilbert weights} $s(m,\mathbf{c}_i)$ of $X$ in $\PP^{n}_\kk$, $n = h^0(X,L) - 1$, with respect to the weight vectors $\mathbf{c}_i$ determined by the filtrations of $\H^0(X,mL)$ induced by the divisors $D_i$.  One description of the Hilbert weights, is given in \cite[Section 4]{Grieve:chow:approx}.

\begin{corollary}\label{ARV-filt-prop2-cor}
Let $L$ be a very ample line bundle on $X$.  Then:
$$
\mathbb{E}(m;\sigma;\mathbf{a}) \geq \min_i \left( \frac{s(m,\mathbf{c}_i)}{m h^0(X, mL) } \right) \text{.}
$$
\end{corollary}
\begin{proof}
We simply observe:
$$ m \alpha(mL,D_i) = \frac{s(m,\mathbf{c}_i)}{h^0(X,mL)}. $$
\end{proof}

\np  Fix a nonzero section 
$$0 \not = s \in \H^0(X,mL)$$ and consider the finite set 
$$\K = \K_{\sigma,\mathbf{a},s}$$ 
that consists of minimal elements of the set:
$$
\left\{
\mathbf{b} \in \NN^{\# \sigma} : \sum_{i \in \sigma} a_i b_i \geq b \mu_{\mathbf{a}}(s) 
\right\}.
$$

\np  In this notation, we then have that:
$$
L^{\otimes m}  \otimes \Ish(\mu_{\mathbf{a}}(s)) = \sum_{\mathbf{b} \in \K} \left( m L - \sum_{i \in \sigma} b_i D_i \right).
$$

\np  One other important point to observe is that, as is a consequence of \cite[Proposition 4.18]{Ru:Vojta:2016}, 
\begin{equation}\label{jet:prop:eqn15}
\div(s) \geq \bigwedge_{\mathbf{b} \in \K} \left( \sum_{i \in \sigma} b_i D_i \right).
\end{equation}
In \eqref{jet:prop:eqn15}, it is meant that we identify the Cartier divisor $\operatorname{div}(s)$ with its pullback to a normal proper model of $X$ that realizes the right hand side as a Cartier divisor.

\np  Furthermore, \cite[Lemma 6.8]{Ru:Vojta:2016}, which uses \eqref{jet:prop:eqn15}, establishes that:
\begin{equation}\label{jet:prop:eqn16}
\bigvee_{\substack{\sigma \in \Sigma \\ \mathbf{a} \in \Delta_\sigma}}
\left(
\sum_{s \in \mathbf{B}_{\sigma;\mathbf{a}}} \div(s) 
\right) 
\\
\geq \frac{b}{b+d} 
\left(
\min_{1 \leq i \leq q} \sum_{\ell \geq 1}^\infty h^0(X, m L - \ell D_i) 
\right) 
D \text{,}
\end{equation}
for $d = \dim X$, the dimension of $X$.
Similarly, in \eqref{jet:prop:eqn16}, we also identify the Cartier divisor $D$ with its pullback to a normal proper model of $X$ that realizes the lefthand side as a Cartier divisor.

\section{The Arithmetic General Theorem}\label{Arithmetic:General:Theorem:Proof}

In this section, we study and extend the Arithmetic General Theorem of Ru and Vojta, \cite{Ru:Vojta:2016}.  In doing so, we prove Theorem \ref{general:Arithmetic:General:Theorem}.  In Section \ref{arithmetic:application:proofs}, we note that it implies Theorem \ref{arithmetic:general:thm:volume:constant}.

\np  In proving Theorem \ref{general:Arithmetic:General:Theorem} below, our argument is based on \cite[proof of Arithmetic General Theorem]{Ru:Vojta:2016}. 

\begin{theorem}[Compare with {\cite[Arithmetic General Theorem]{Ru:Vojta:2016}}]\label{general:Arithmetic:General:Theorem}
Let $X$ be a geometrically irreducible projective variety over $\KK$, a number field or characteristic zero function field, and let $D_1,\dots,D_q$ be nonzero effective Cartier divisors on $X$, defined over $\FF$, which intersect properly.  Let $L$ be a big line bundle on $X$, defined over $\KK$.  Then, for each $\epsilon > 0$, there exists constants $a_\epsilon$, $b_\epsilon > 0$ with the property that either:
$$ h_L(x) \leq a_\epsilon;$$
or
$$ 
m_S(x,D) \leq \left( \max_{1 \leq j \leq q} \gamma(L,D_j) + \epsilon \right) h_L(x) + b_\epsilon
$$
for all $\KK$-rational points $x \in X(\KK)$ outside of some proper Zariski closed subset $Z \subsetneq X$.
\end{theorem}

\np  In our proof of Theorem \ref{general:Arithmetic:General:Theorem}, to reduce notation, we denote, with extensive abuse of notation, by $X$ the base change of $X$ with respect to the field extension $\FF / \KK$.  We also denote by $L$ the pullback of $L$ via this base change.  When no confusion is likely, we make no explicit mention about allowing coefficients in $\FF$, as opposed to only allowing coefficients in $\KK$.  We also emphasize that our proof relies on Proposition \ref{ARV-filt-prop2}.  For example, it requires that the divisors $D_1,\dots,D_q$, which are defined over $\FF$, intersect properly (over $\FF$).

\np  We now proceed to prove our Main Arithmetic General Theorem.

\begin{proof}[Proof of Theorem \ref{general:Arithmetic:General:Theorem}]
Let $d$ denote the dimension of $X$.  Let $\epsilon > 0$ and choose a real number $\epsilon_2 > 0$ together with positive integers $m$ and $b$ so that:
\begin{equation}
\left( 1 + \frac{d}{b} \right) \max_{1 \leq i \leq q} \left( \frac{m h^0(X,mL) + m \epsilon_2 }{ \sum_{\ell \geq 1} h^0(X,mL - \ell D_i) }\right) < 
\max_{1 \leq i \leq q} \gamma(L,D_i) + \epsilon.
\end{equation}

Write:
$$
\bigcup_{\sigma; \mathbf{a}} \mathcal{B}_{\sigma;\mathbf{a}} = \mathcal{B}_1 \bigcup \dots \bigcup \mathcal{B}_{T_2} = \{s_1,\dots,s_{T_2} \}.
$$
For each $i = 1,\dots, T_1$, let 
$$J_i \subseteq \{1,\dots,T_2\}$$ be the subset such that 
$$\mathcal{B}_i = \{s_j : j \in J_i \}.$$  
Let $B_i$ denote the divisor 
$$B_i := \div\left( \sum_{j \in B_i} \div(s_j) \right);$$ 
choose Weil functions $\lambda_{D,v}(\cdot)$, $\lambda_{B_i,v}(\cdot)$ and $\lambda_{s_j,v}(\cdot)$, for each $v \in S$, and each $j = 1,\dots, T_2$.

The key point then is that, because of \eqref{jet:prop:eqn16} and \cite[Proposition 4.10]{Ru:Vojta:2016} (parts (b) and (c) can be adapted using the general theory of \cite[Chapter 10]{Lang:Diophantine} so as to apply to our current setting), for each $v \in S$, it also holds true that:
$$
\max_{1 \leq i \leq T_1} \lambda_{B_i,v}(\cdot) + \mathrm{O}_v(1) \geq
\frac{b}{b+d}\left(
\min_{1 \leq i \leq q} \sum_{\ell \geq 1} h^0(X, mL - \ell D_i) \right) \lambda_{D,v}(\cdot).
$$
But:
$$
\max_{1 \leq i \leq T_1} \lambda_{B_i,v}(\cdot) + \mathrm{O}_v(1) = \max_{1 \leq i \leq T_1} \sum_{j \in J_i} \lambda_{s_j,v}(\cdot) + \mathrm{O}_v(1)
$$
and so:
$$
\frac{b}{d+b} \left(  \min_{1 \leq i \leq q} \sum_{\ell \geq 1} h^0(X, mL - \ell D_i ) \right) \lambda_{D,v}(\cdot) \leq \max_{1 \leq i \leq T_1} \sum_{j \in J_i} \lambda_{s_j,v}(\cdot) + \mathrm{O}_v(1).
$$
We now apply Schmidt's Subspace Theorem (in the form of Proposition \ref{Schmidt:Subspace:Thm}).  In particular, we may apply that remark to the sections $s_j$ since they are sections of $L$ (a priori with coefficients in $\FF$) and since the big line bundle $L$ is defined over the base field $\KK$.  

Our conclusion then is that there exists a proper Zariski closed subset $Z \subsetneq X$ and constants $a_{\epsilon_2}$, $b_{\epsilon_2}$ so that for all $\KK$-points $x \in X \setminus Z$ either:
$$
h_{mL}(x) \leq a_{\epsilon_2}
$$
or 
$$
\sum_{v \in S} \max_J \sum_{j \in J} \lambda_{s_j,v}(x) \leq \left(h^0(X,mL) + \epsilon_2 \right) h_{mL}(x) + b_{\epsilon_2}.
$$
Here, the maximum is taken over all subsets 
$$J \subseteq \{1,\dots, T_2 \}$$
 for which the sections $s_j$, for $j \in J$, are linearly independent.

In particular, either:
$$
h_{mL}(x) \leq a_{\epsilon_2}
$$
or
$$
\sum_{v \in S} \lambda_{D,v}(x) \leq \left(1 + \frac{d}{b} \right) \max_{1 \leq i \leq q} \left( \frac{h^0(X,mL) + \epsilon_2} {\sum_{\ell \geq 1} h^0(X,mL- \ell D_i) 
}
\right) h_{mL}(x) + b_{\epsilon_2}'
$$
for all $\KK$-points $x \in X \setminus Z$.  Finally, since 
$$h_{mL} (x) = m h_L(x),$$ we obtain that either:
$$
h_L(x) \leq A_\epsilon
$$
or
$$
\sum_{v \in S} \lambda_{D,v}(x) \leq \left( \max_{1 \leq i \leq q} \gamma(L,D_i) + \epsilon \right) h_L(x) + B_\epsilon
$$
for positive constants $A_\epsilon$, $B_\epsilon$ (depending on $\epsilon$) and all $\KK$-points $x \in X \setminus Z$.
\end{proof}

\section{Proof of Theorem \ref{arithmetic:general:thm:volume:constant}, Corollary \ref{Roth:exceptional:divisor:intro} and Corollary \ref{vojta:inequalities}
}\label{arithmetic:application:proofs}

In this section, we prove Theorem \ref{arithmetic:general:thm:volume:constant}, Corollary  \ref{Roth:exceptional:divisor:intro} and Corollary   \ref{vojta:inequalities}.

\np  First we prove Theorem \ref{arithmetic:general:thm:volume:constant}. 

\begin{proof}[Proof of Theorem \ref{arithmetic:general:thm:volume:constant}]
Combine Theorem \ref{general:Arithmetic:General:Theorem} and Proposition \ref{jet:prop:1}.
\end{proof}

\np  Next, we prove Corollary  \ref{Roth:exceptional:divisor:intro}.    
\begin{proof}[
Proof of Corollary  \ref{Roth:exceptional:divisor:intro}]   Corollary  \ref{Roth:exceptional:divisor:intro} is a consequence of Theorem \ref{arithmetic:general:thm:volume:constant} applied to $\widetilde{X}$.  
\end{proof}

\np  Finally, we note that 
Corollary  \ref{vojta:inequalities} is a special case of Theorem \ref{vojta:inequalities:thm} below.

\begin{theorem}\label{vojta:inequalities:thm}
Suppose that $X$ is a geometrically irreducible $\QQ$-Gorenstein geometrically normal variety, defined over $\KK$, and assume that the $\QQ$-Cartier divisor $-\K_X$ is $\QQ$-ample.  Let $Y \subsetneq X$ be a proper subscheme, defined over $\FF$, and write the exceptional divisor $E$ of 
$$\pi \colon \widetilde{X} \rightarrow X_\FF,$$ 
the blowing-up of $X$ along $Y$, in the form 
$$E = E_1 + \hdots + E_q,$$ 
for $E_i$ effective Cartier divisors on $\widetilde{X}$.  In this context, assume that the divisors $E_1,\dots,E_q$ intersect properly and
$$
\beta(- \pi^*\K_X, E_i) \geq 1
$$
for all $i$.  Then Vojta's inequalities hold true.  Precisely, if $M$ is a big line bundle on $X$, defined over $\KK$, then for all $\epsilon > 0$, there exists constants $a_\epsilon, b_\epsilon > 0$ so that either:
$$
h_{-\K_X}(x) \leq a_\epsilon
$$
or
$$
\sum_{v \in S} \lambda_{Y,v}(x) + h_{\K_X}(x) \leq \epsilon h_M(x) + b_\epsilon
$$
for all $\KK$-rational points $x \in X(\KK)$ outside of some proper subvariety of $X$.
\end{theorem}

\begin{proof}[Proof of Theorem \ref{vojta:inequalities:thm} and Corollary  \ref{vojta:inequalities}]  
Without loss of generality, we may assume that the ample Cartier divisor $-\K_X$ is integral.  We first establish the well-known reduction step.  If Vojta's inequalities hold true for $M = - \K_X$ and all $\epsilon > 0$, then the same is true for all big line bundles $M$ (perhaps after adjusting $\epsilon$ and the proper subvariety).  

To this end, as in \cite[Corollary  2.2.7]{Laz}, since $- \K_X$ is ample and $M$ is big, there exists $m > 0$ and an effective line bundle $N$ on $X$ so that:
$$
m M \sim_{\mathrm{lin}} - \K_X + N.
$$
Thus:
$$
m h_M(\cdot) = h_{-\K_X}(\cdot) + h_N(\cdot) + \mathrm{O}(1);
$$
and:
\begin{equation}\label{eqn7}
h_{-\K_X}(\cdot) \leq m h_M(\cdot) + \mathrm{O}(1)
\end{equation}
outside of $\mathrm{Bs}(N)$, the base locus of $N$.  

Let $\epsilon > 0$.  
We show that:
\begin{equation}\label{eqn8}
\sum_{v \in S} \lambda_{Y,v} (x) + h_{\K_X}(x) \leq \epsilon h_M(x) + \mathrm{O}(1)
\end{equation}
for all $x \in X(\KK)$ outside of some proper Zariski closed subset $Z \subsetneq X$.  
Note:
\begin{equation}\label{eqn9}
\epsilon h_{-\K_X}(\cdot) \leq \epsilon m h_M(\cdot) + \mathrm{O}(1)
\end{equation}
outside of $\mathrm{Bs}(N)$;
put:
$$
\epsilon' = \frac{\epsilon}{m}.
$$
We can rewrite \eqref{eqn9} as:
\begin{equation}\label{eqn11}
\epsilon' h_{-\K_X}(\cdot) \leq \epsilon h_M(\cdot) + \mathrm{O}(1)
\end{equation}
outside of $\mathrm{Bs}(N)$.  On the other hand, by assumption:
\begin{equation}\label{eqn12}
\sum_{v \in S} \lambda_{Y,v} (x) + h_{\K_X}(x) \leq \epsilon' h_{-\K_X}(x) + \mathrm{O}(1)
\end{equation}
for all $x \in X(\KK)$ outside of some Zariski closed subset $W \subsetneq X$ depending on $\epsilon'$.  Combining \eqref{eqn11} and \eqref{eqn12}, we then have:
\begin{equation}\label{eqn12'}
\sum_{v \in S} \lambda_{Y,v}(\cdot) + h_{\K_X}(x) \leq \epsilon' h_{-\K_X}(x) + \mathrm{O}(1) \leq \epsilon h_M(x) + \mathrm{O}(1),
\end{equation}
for all $x \in X(\KK)$ outside of
$$Z := W \bigcup \mathrm{Bs}(N). $$
The inequality \eqref{eqn8} is thus implied by \eqref{eqn12'}.

Now put $M = - \K_X$ and let $\epsilon > 0$.    We can rewrite the inequality
$$
\sum_{v \in S} \lambda_{Y,v}(\cdot) + h_{\K_X}(\cdot) \leq \epsilon h_{-\K_X}(\cdot) + \mathrm{O}(1)
$$
in terms of the blowing-up 
$$\pi\colon \widetilde{X} \rightarrow X_\FF$$ 
of $X$ along $Y$:  
\begin{equation}\label{eqn14}
m_S(\cdot,E) \leq (1 + \epsilon) h_{-\pi^* \K_X}(\cdot) + \mathrm{O}(1).
\end{equation}
By assumption, 
$$E = E_1 + \hdots + E_q,$$
the Cartier divisors $E_i$ intersect properly and
$$ 
\beta(-\pi^* \K_X,E_i) \geq 1,
$$
for all $i$.  Thus:
$$
\max_{1 \leq i \leq q} \beta(-\pi^* \K_X, E_i)^{-1} + \epsilon \leq 1 + \epsilon.
$$

We now apply Theorem \ref{general:Arithmetic:General:Theorem} and Proposition \ref{jet:prop:1}.  There exists constants $\tilde{a}_\epsilon, \tilde{b}_\epsilon > 0$ so that either:
\begin{equation}\label{eqn15}
h_{- \pi^* \K_X}(x) \leq \tilde{a}_\epsilon;
\end{equation}
or
\begin{equation}\label{eqn16}
m_{S}(x,E) \leq \left( \max_{1 \leq i \leq q } \beta(- \pi^* \K_X, E_i)^{-1} + \epsilon \right) h_{- \pi^* \K_X}(x) + \tilde{b}_\epsilon
\end{equation}
for all $\KK$-rational points $x \in \widetilde{X}(\KK)$ outside of some proper Zariski closed subset $\widetilde{Z} \subsetneq \widetilde{X}$.  

We can rewrite these inequalities \eqref{eqn15} and \eqref{eqn16} in terms of $X$.   As in \eqref{eqn12'} and \eqref{eqn14}, we obtain constants $a_\epsilon, b_\epsilon > 0$ so that either:
$$
h_{-\K_X}(x) \leq a_\epsilon
$$
or 
$$
\sum_{v \in S} \lambda_{Y,v}(x) \leq  ( 1 + \epsilon)h_{-  \K_X}(x) + b_\epsilon
$$
for all $\KK$-rational points $x \in X(\KK)$ outside of some proper subvariety $Z \subsetneq X$.
\end{proof}

\end{document}